\documentclass[a4paper,10pt]{article}
\usepackage{amsmath,amssymb}
\usepackage[utf8]{inputenc}
\usepackage[dvips]{graphics}
\usepackage[dvips]{epsfig}
\usepackage{mathrsfs}
\usepackage{tikz}

\newcommand{\dend}{\operatorname{Dend}}
\newcommand{\bt}{\mathbb{Y}}
\newcommand{\nct}{\operatorname{NCT}}
\newcommand{\ncp}{\operatorname{NCP}}

\newcommand{\ZZ}{\mathbb{Z}}
\newcommand{\QQ}{\mathbb{Q}}

\newcommand{\TA}{\mathbb{A}}

\newcommand{\Id}{\operatorname{Id}}

\newcommand{\tm}{\mathbb{M}}
\newcommand{\ms}{\mathcal{M}}

\newcommand{\ps}{\mathscr{P}}
\newcommand{\proj}{\mathsf{P}}
\newcommand{\under}{\backslash}
\renewcommand{\over}{/}

\newcommand{\Y}{\epsfig{file=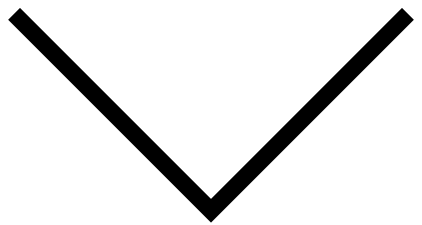,height=2mm}}
\newcommand{\lft}{\epsfig{file=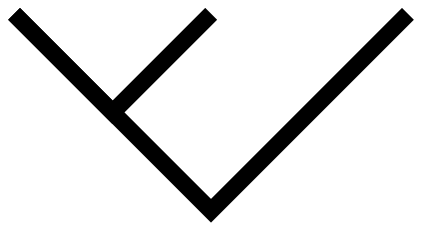,height=2mm}}
\newcommand{\rgt}{\epsfig{file=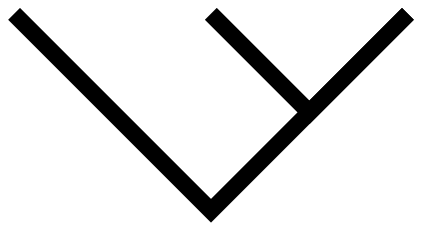,height=2mm}}

\newcommand{\gA}{\epsfig{file=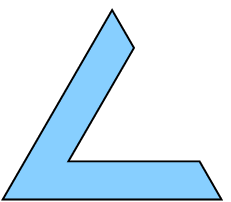,height=2.5mm}}
\newcommand{\dA}{\epsfig{file=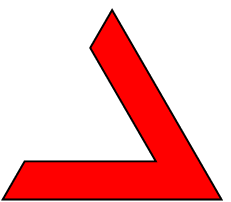,height=2.5mm}}
\newcommand{\mA}{\epsfig{file=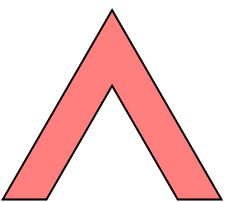,height=2.5mm}}

\newcommand{\gAs}{\epsfig{file=Garbre.eps,height=1.5mm}}
\newcommand{\dAs}{\epsfig{file=Darbre.eps,height=1.5mm}}
\newcommand{\mAs}{\epsfig{file=Marbre.eps,height=1.5mm}}

\newcommand{\hmin}{\widehat{0}}
\newcommand{\hmax}{\widehat{1}}

\newtheorem{theorem}{Theorem}[section] 
\newtheorem{proposition}[theorem]{Proposition} 
 
\newtheorem{corollary}[theorem]{Corollary} 
\newtheorem{lemma}[theorem]{Lemma} 
\newtheorem{definition}{Definition}

\newenvironment{proof}{\begin{trivlist}\item{\bf{Proof.}}}
  {\hfill\rule{2mm}{2mm}\end{trivlist}}

\title{Some dendriform functors}
\author{F. Chapoton}
\date{\today}

\begin{document}

\maketitle

\begin{abstract}
  We make a first step towards categorification of the dendriform
  operad, using categories of modules over the Tamari lattices. This
  means that we describe some functors that correspond to part of the
  operad structure.
\end{abstract}

\section{Introduction}

The notion of dendriform algebra is one of several new kinds of
algebras introduced by Jean-Louis Loday around 2000. In brief, a
dendriform algebra is an associative algebra, together with a
decomposition of the associative product as the sum of two binary
products, with some appropriate axioms. Loday has proved that the free
dendriform algebras can be described using classical combinatorial
objects called planar binary trees.

Some of the most interesting properties of these algebras can be
formulated in terms of the corresponding operad $\dend$. Loday has
shown that the operad $\dend$ is Koszul. Later, it was
proved in \cite{anticyclic} that the operad $\dend$ is anticyclic.

There has been several hints that the dendriform structures (algebras
and operad) are closely related to a natural partial order on planar
binary trees, called the Tamari poset. First, there is an associative
product on free dendriform algebras, and the product of two planar
binary trees can be described as an interval in a Tamari poset \cite{lr_order}. Next,
the anticyclic structure of the operad $\dend$ is given by a
collection of matrices than can be described directly from the Tamari
poset \cite{coxeter_tamari}.

The global aim of this article would be to categorify the operad
$\dend$. An operad is essentially a collection of Abelian groups and
linear maps between them. One has to define Abelian categories and
functors between them, in such a way that the Grothendieck group
construction recovers the initial operad.

Such a categorification has already been obtained in
\cite{diassociative} for the Koszul dual cooperad of the dendriform
operad. This dual cooperad describes dialgebras, which is another one
of Loday's kinds of algebras. This was though a much simpler
situation, involving only quivers of type $\TA$ and their products.

For the operad $\dend$, we will only present here some partial
results. The starting idea is to use the category of modules over the
Tamari poset as a categorification for the Abelian group spanned by
planar binary trees. In this article, we have obtained functors
corresponding to the following linear maps:
\begin{itemize}
\item  the first composition map $\circ_1 : \dend(m)\otimes \dend(n)\to \dend(m+n-1)$,
\item  the associative product $* : \dend(m)\otimes \dend(n)\to \dend(m+n)$,
\item  a new associative product $\# : \dend(m)\otimes \dend(n)\to \dend(m+n-1)$.
\end{itemize}
The product $\#$ has been explained to the author by Jean-Christophe
Aval and Xavier Viennot, in the setting of Catalan alternative
tableaux, see \cite{aval_viennot_preprint} and \cite{viennot_fpsac07}
for related works.

We identify the usual basis of the dendriform algebra and operad,
indexed by planar binary trees, with the basis of the Grothendieck
groups of the Tamari posets coming from simple modules. There is
another basis of the Grothendieck groups, coming from projective
modules. This will be called the basis of projective elements.

One important tool in this article is a small set-operad contained in
$\dend$, introduced in \cite{moules}. This sub-operad can be described using
noncrossing configurations in a regular polygon. We will in
particular show that the basis of projective elements is contained in
this sub-operad.

The reason why we have only partial results is the following : the
other composition maps of the operad $\dend$ do not preserve the set
of projective elements. This makes more difficult to define the
corresponding functors, even if it is possible to guess what they
should be.

\section{General setup}

\subsection{Planar binary trees}

Let us first introduce very classical combinatorial objects, called
\textbf{planar binary trees}. They can be concisely defined as
follows: a planar binary tree is a either a dot $\circ$ or a pair of planar
binary trees $(x\,y)$.

This leads to a representation as a sequence of dots and parentheses,
such as
\begin{equation}
  (((\circ \circ) (\circ (\circ \circ))) (\circ ((\circ \circ) \circ))). 
\end{equation}

These objects can be converted into planar trees in a simple way. For
the previous expression, the result is depicted in Figure
\ref{expl_bt}. The $\circ$ elements become leaves of the tree.
Vertices of the tree correspond to pairs of matching parentheses. The
outermost pair of parentheses correspond to the root vertex.

\begin{figure}
  \begin{center}
    \epsfig{file=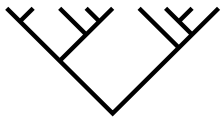,height=1.5cm} 
    \caption{A planar binary tree of degree $8$}
    \label{expl_bt}
  \end{center}
\end{figure}

We will always draw planar binary trees with their leaves at the top,
on an horizontal line, and their root at the bottom.

Let us define the degree of a planar binary tree to be the number of
vertices or the number of leaves minus one. Let $\bt_n$ be the set of
planar binary trees of degree $n$. For example, $\bt_1=\{\Y\}$ and
$\bt_2=\{\lft,\rgt\}$. The cardinal of $\bt_n$ is the Catalan number
$\frac{1}{n+1}\binom{2n}{n}$. Let $\bt$ be the union of the sets
$\bt_n$ for $n\geq 1$.

Let us now recall two basic combinatorial operations on $\bt$: the
over product $\over$ and the under product $\under$, both associative
and graded.

Let $x,y$ be planar binary trees. The planar binary tree $x\over y$ is
obtained by identifying the root of $x$ with the leftmost leaf of $y$.
Similarly $x\under y$ is obtained by identifying the root of $y$ with
the rightmost leaf of $x$.

For example, $\Y \over \Y = \lft$ and $\Y \under \Y =\rgt$.

Note that one has $(x\over y) \under z=x\over (y \under z )$.

\begin{lemma}
  \label{basic_alt}
  Any planar binary tree of degree at least $2$ can either be written
  $x\over y$ for some planar binary trees $x,y$ or $\Y \under x$ for
  some planar binary tree $x$.
\end{lemma}
\begin{proof}
  If there is at least one vertex to the left of the root, then
  the first decomposition is possible. Else one must be in the second case.
\end{proof}

There is an obvious involution on planar binary tree, the (left-right)
reversal, that will be denoted $x \mapsto \overline{x}$. It exchanges
the over and under products : $\overline{x \over y}=\overline{y}
\under \overline{x}$

\subsection{Tamari poset}

\label{tamari_poset}

There is a partial order on the set $\bt_n$, called the Tamari poset.
It was introduced by Dov Tamari in \cite{dov_tamari} and proved there to be a
lattice.

\begin{figure}
  \begin{center}
    \epsfig{file=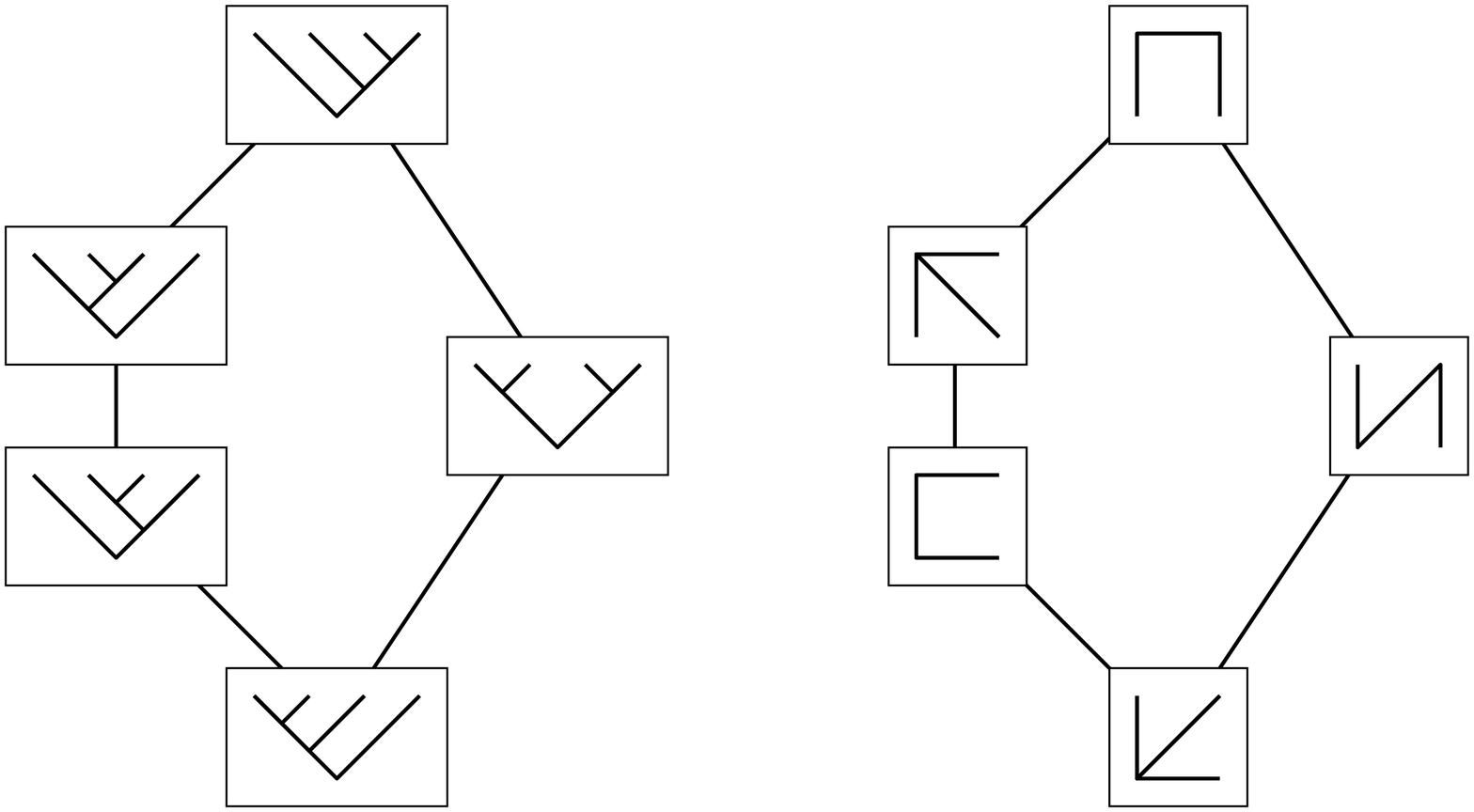,height=4cm} 
    \caption{Left: Tamari poset \quad Right: Good pivots between projective elements}
    \label{pentago}
  \end{center}
\end{figure}

This partial order can be defined as the transitive closure of
elementary moves: $x \leq y$ if $x$ is obtained from $y$ by a sequence
of local changes, replacing the configuration $\rgt$ by the
configuration $\lft$ somewhere in the tree.

The elementary moves can be described more formally as follows. 
\begin{itemize}
\item For any planar binary trees $a,b,c$, there is an elementary move
  from $a\over \Y \under (b\over \Y \under c)$ to $(a\over \Y \under
  b)\over \Y \under c$.
\item If $x \to y$ is an elementary move, then $x\over z \to y \over
  z$, $z\over x \to z \over y$ and $\Y \under x \to \Y \under y$ are
  also elementary moves.
\end{itemize}

In $\bt_n$, there is a unique minimum element $\hmin$ , which is the
left comb of degree $n$, and a unique maximum element $\hmax$, which
is the right comb of order $n$.

We will use the following convention: the Hasse diagram of the Tamari
poset is drawn with its maximum element at the top and its minimum
element at the bottom.  We will orient the edges of the Hasse diagram
in the decreasing way (from top to bottom).

It is well-known that the Hasse diagram of the Tamari poset is the
skeleton of a simple polytope, the associahedron or Stasheff polytope.

\begin{proposition}
  The Hasse diagram of the Tamari poset $\bt_n$ is a regular graph of
  order $n-1$.
\end{proposition}

The reversal is an anti-automorphism of the Tamari poset: $x\leq y
\Leftrightarrow \overline{y} \leq \overline{x}$.

\smallskip In this article, we will consider the Hasse diagram of the
Tamari poset as a quiver with relations. The arrows are the edges with
the decreasing orientation. Relations are given by all possible
equalities between paths. The category of representations of this
quiver with relations, which is equivalent to the category of modules
over the incidence algebra of the Tamari poset, will be denoted by
$\mod \bt_n$.

\subsection{Dendriform operad}

Let $\dend(n)$ be the free Abelian group $\ZZ \bt_n$. Then the
collection $(\dend(n))_{n \geq 1}$ can be given the structure of
an operad in the category of Abelian groups, called the Dendriform
operad \cite{loday_lnm}.

More precisely, the collection $(\dend(n))_{n \geq 1}$ is a
non-symmetric operad. This means that for all $m\geq 1$, $n \geq 1$
and for each $1 \leq i \leq m$, there is a linear map $\circ_i$ from
$\dend(m)\otimes \dend(n) \to \dend(m+n-1)$.  All these maps satisfy
the ``associativity'' conditions, \textit{i.e.} the operad
axioms. Furthermore, the distinguished element $\Y$ in $\dend(1)$
plays the role of a unit for the composition maps $\circ_i$.

There is a precise combinatorial description of the maps $\circ_i$, as
some kind of double shuffle of planar binary trees. As this
is not needed in the sequel, we will not recall it there.

The reversal map sends compositions to compositions:
\begin{equation}
  \label{compo_reverse}
  \overline{x \circ_i y} = \overline{x} \circ_{m+1-i} \overline{y},
\end{equation}
for $x\in \dend(m)$.

\subsection{Dendriform algebra}

Let $\dend$ be the direct sum of all Abelian groups $\dend(n)$ for $n
\geq 1$. Then there is an associative graded product $*$ on $\dend$,
which can be defined using the operad structure as follows:
\begin{equation}
  \label{def_star}
  x * y = ((\lft+\rgt) \circ_2 y) \circ_1 x.
\end{equation}

More concretely, by a result of Loday and Ronco in \cite{lr_order},
this associative product is also given by
\begin{equation}
  x * y =\sum_{ x \over y \leq z \leq x\under y } z.
\end{equation}

The product can also be informally described as follows: the product
of two planar binary trees $x$ and $y$ is the sum over all planar binary trees obtained by
shuffling the right side of $x$ with the left side of $y$.

For instance, one has $\Y * \Y=\lft+\rgt$.

\begin{proposition}
  The reversal of planar binary trees is an
  anti-automorphism of the $*$ product: $\overline{x * y}=\overline{y}
  * \overline{x}$.
\end{proposition}
\begin{proof}
  One uses definition (\ref{def_star}), the fact that the reversal
  maps $\circ_1$ to $\circ_n$ by (\ref{compo_reverse}), and the
  ``commutativity'' axiom of operads.
\end{proof}

\medskip

Let us note that the over and under products can also be expressed
using the composition maps of the Dendriform operad:
\begin{equation}
  \label{def_over}
  x \over y = (y \circ_1 \lft)\circ_1 x
\end{equation}
and
\begin{equation}
  \label{def_under}
  x \under y = (x \circ_m \rgt)\circ_{m+1} y,
\end{equation}
if $x \in \bt_m$.

Let us give now a few useful relations.

\begin{lemma}
  One has 
  \begin{align}
  \label{star_over}
    x \over (y*z)&=(x\over y)*z,\\
\label{circ_over}
    (x \circ_1 y)\over z&=(x\over z)\circ_1 y,\\
\label{star_circ}
(x \circ_1 y)*z&=(x*z)\circ_1 y.
  \end{align}
\end{lemma}
\begin{proof}
  This is an exercise in the Dendriform operad, using
  (\ref{def_star}), (\ref{def_over}) and (\ref{def_under}).
\end{proof}

Let us introduce a bilinear form on $\dend(n)$. Let $x=\sum_{s\in\bt_n} x_s s$
and $y=\sum_{t\in\bt_n} y_t t $ be elements of $\dend(n)$. One defines
\begin{equation}
  \langle x , y \rangle = \sum_{s\leq t\in \bt_n} x_s y_t \mu(s,t),
\end{equation}
where $\mu$ is the Möbius function of the Tamari poset. The Möbius
function $\mu$ is known to have values in $\{-1,0,1\}$ (see
\cite{pallo,BjWa2}).

This bilinear form is called the Euler form. Note that this is not
symmetric. The Euler form has a natural meaning in representation
theory, namely it comes from the alternating sum of dimensions of Ext
groups in the category $\mod \bt_n$.

Let $E$ be the associated quadratic form $E(x)=\langle x,x \rangle$.

\subsection{Anticyclic structure}

The operad $\dend$ is in fact an anticyclic operad. This means that
there exists, for each $n\geq 1$, a linear endomorphism $\tau$ of
$\dend(n)$ satisfying $\tau(\Y)=-\Y$, $\tau^{n+1}=\Id$
and the following compatibility conditions with the composition maps
$\circ_i$ of the Dendriform operad:
\begin{align}
\label{anticyclic0}  \tau(x \circ_n y)&=-\tau(y)\circ_1 \tau(x),\\
\label{anticyclic1}  \tau(x \circ_i y)&=\tau(x)\circ_{i+1} y \quad \text{ if }1\leq
  i<n, 
\end{align}
where $x \in \dend(n)$ and $y\in\dend(m)$.

We have proved in \cite{coxeter_tamari} that there exists a linear endomorphism
$\theta$ of $\dend(n)$ such that $\tau=(-1)^n \theta^2$ and with the
following properties:
\begin{align}
\label{theta0}    \theta(\Y)&=-\Y,\\
\label{theta1}    \theta(x \backslash y)&=-\theta(x) * \theta(y),\\
\label{theta2}    \theta(x * y)&=-\theta(x) / \theta(y),\\
\label{theta3}    \theta^{-1}(x / y)&=-\theta^{-1}(x) * \theta^{-1}(y),\\
\label{theta4}    \theta^{-1}(x * y)&=-\theta^{-1}(x) \backslash \theta^{-1}(y). 
\end{align}

The endomorphism $\theta$ is defined in \cite{coxeter_tamari} using
only the Tamari partial order on $\bt_n$. In fact, $\theta$ has a
natural meaning in representation theory: it comes from the
Auslander-Reiten translation, which is an auto-equivalence of the
derived category of $\mod \bt_n$. It follows from this definition and
the fact that reversal is an anti-automorphism of the Tamari poset
that $\theta$ has the following property:
\begin{equation}
\label{theta_reverse}
  \theta(\overline{x})=\overline{\theta^{-1}(x)}.
\end{equation}

\section{The operad of noncrossing plants}

\subsection{Noncrossing combinatorics}

Let us now introduce other combinatorial objects: noncrossing trees
and noncrossing plants.

Let $n\geq 1$. Let $O_n$ be a convex polygon with $n+1$ vertices, with
a distinguished side called the \textbf{base side}, that we will use
as bottom side. If $n\geq 2$, the other sides are numbered from $1$ to
$n$ in the clockwise order. If $n=1$, by convention, there is only one
side, which is the base side and the side $1$.

A \textbf{noncrossing tree} of degree $n$ is a subset of the set of
edges between vertices of the polygon $O_n$ such that
\begin{itemize}
\item No two edges cross (but they can meet at their end points),
\item There is no cycle,
\item The collection is maximal with respect to these properties.
\end{itemize}

Let $\nct_n$ be the set of noncrossing trees in $O_n$.

Remark: noncrossing trees can be identified with exceptional
collections (up to permutation) in the derived category of the quiver
of type $\TA$ (see \cite{araya}).

\begin{figure}
  \begin{center}
    \epsfig{file=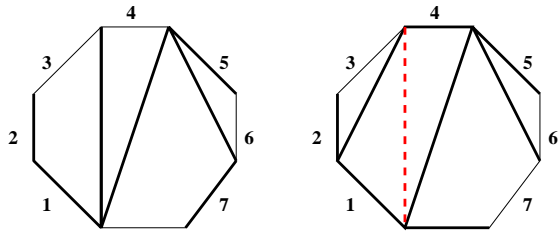,height=3cm} 
    \caption{A noncrossing tree and a noncrossing plant in $O_7$}
    \label{expl_nct}
  \end{center}
\end{figure}

A \textbf{noncrossing plant} of degree $n$ is a disjoint pair of
subsets of the set of edges between vertices of the polygon $O_n$,
called numerator edges and denominator edges, such that
\begin{itemize}
\item No two edges cross (but they can meet at their end points),
\item Each cycle of denominator edges surrounds exactly one numerator edge.
\item Each numerator edge is inside a cycle of denominator edges.
\item The collection is maximal with respect to these properties.
\end{itemize}

Let $\ncp_n$ be the set of noncrossing plants in $O_n$.

Note that noncrossing trees can be identified with noncrossing plants
without numerator edges. 

Remark: the numerator edges cannot be on the boundary of the polygon $O_n$.

\medskip

An \textbf{angle} in a noncrossing tree is a pair of edges with a
common endpoint $v$ that are adjacent in the ordered set of edges
incident to $v$.

The $3$ vertices involved in an angle define a triangle. Define the
base side of this triangle to be the edge which is closest to the base
side of the polygon $O_n$. An angle is said to of type $\gA,\mA$ or
$\dA$ according to the noncrossing tree obtained by restriction to the
triangle.

Let $N_{\gAs}(P)$, $N_{\mAs}(P)$ and $N_{\dAs}(P)$ be the numbers of angles of type $\gA$, $\mA$ and $\dA$ in a noncrossing tree $P$.

\begin{lemma}
  In any noncrossing tree $P$ of degree $n$, one has $n-1$ angles:
  $N_{\gAs}(P)+N_{\mAs}(P)+N_{\dAs}(P)=n-1$.
\end{lemma}
\begin{proof}
  This is an easy combinatorial exercise.
\end{proof}

\subsection{Operad structure}

In the article \cite{moules}, it was proved that the smallest
sub-operad (in the category of sets) of $\dend$ containing the
elements $\lft,\lft+\rgt,\rgt$ of $\dend(2)$ can be described by
noncrossing plants.

The composition of noncrossing plants is given by gluing as follows.
Let $P\in\ncp_m$ and $Q\in \ncp_n$. Let us describe $P\circ_i Q$ as a
noncrossing plant in $O_{m+n-1}$. First consider the polygon obtained
by identification of the base side of $O_n$ with the side $i$ of
$O_m$. By choosing appropriate deformations of $O_n$ and $O_m$, one
can assume that this polygon is convex and can be identified with
$O_{m+n-1}$, with base side the base side of $O_m$. Then one takes the
union of $P$ and $Q$ inside this glued polygon, with a special care
along the gluing edge. If the gluing edge belongs both to $P$ and $Q$,
it is kept in $P\circ_i Q$. If it belongs to just one of $P$ or $Q$,
then it is not kept in $P\circ_i Q$. If it does not belong to $P$ nor
to $Q$, then it is replaced by a numerator edge in $P\circ_i Q$.

Let us say that a noncrossing tree is \textbf{based} if it contains
the base side.

\begin{lemma}
  \label{pas_plante}
  Let $P$ and $Q$ be noncrossing trees. Then $P\circ_i Q$ is a
  noncrossing tree if and only if $Q$ is based or $P$ contains the
  border side $i$.
\end{lemma}
\begin{proof}
  This follows from the explicit description of the composition maps
  given above. The only case that should be avoided is the case when a
  numerator edge is created during composition.
\end{proof}

The operad of noncrossing plants is generated by $\gA,\mA,\dA$
(\cite[Th. 6.3]{moules}).

The inclusion of the operad of noncrossing plants in $\dend$ is
the unique morphism of operad extending the following map on
generators
\begin{equation}
  \begin{cases}
 \gA \mapsto \lft,\\
 \mA \mapsto \lft+\rgt,\\
 \dA \mapsto \rgt.    
  \end{cases}
\end{equation}

In the context of noncrossing plants, the associative $*$ product
defined by (\ref{def_star}) is given by
\begin{equation}
  \label{prod_ncp}
  P * Q = (\mA \circ_2 Q) \circ_1 P. 
\end{equation}
This can be described as gluing $P$ and $Q$ on the left and right
sides of a triangle $\mA$.

\begin{lemma}
  The set of noncrossing trees is closed under the $*$ product.
\end{lemma}
\begin{proof}
  Indeed, formula \eqref{prod_ncp} shows that the product of
  noncrossing trees only involves compositions that do not create a
  numerator edge.
\end{proof}

\begin{lemma}
  \label{decompo_nct}
  Each noncrossing tree has a unique decomposition as a $*$ product of
  based noncrossing trees.
\end{lemma}
\begin{proof}
  Let $P$ be a noncrossing tree. There is a unique path in $P$ from
  the left vertex of the base side to the right vertex of the base
  side. Each edge of this path can be considered as the base side of a
  based noncrossing tree, by restriction. Then $P$ is the $*$ product
  of these noncrossing trees in their natural order. Uniqueness is
  clear.
\end{proof}

From (\ref{def_over}) and (\ref{def_under}), the over product can be
restated as
\begin{equation}
  P\over Q = (Q \circ_1 \gA)\circ_1 P
\end{equation}
and the under product as
\begin{equation}
  P \under Q = (P \circ_m \dA)\circ_{m+1} Q,
\end{equation}
if $P \in \nct_m$.

One can easily see from these formulas that the set of noncrossing
trees is closed with respect to the over and under products.

\section{First description of projective elements}

We will from now on identify the Grothendieck group of the Tamari
poset $\bt_n$ with the Abelian group $\dend(n)=\ZZ\bt_n$ by sending
the simple module associated with a planar binary tree $T$ to the same
planar binary tree $T$ in the natural basis of $\dend(n)$.

Let $\proj(x)$ be the projective module for the Tamari poset associated with the vertex $x$. More precisely, $\proj(x)$ is the $\bt_n$-module defined at the level of vertices by a copy of $\QQ$ for each element of the interval $\{ y \in \bt_n \mid \hmin \leq y \leq x\}$ and the null vector space elsewhere, and by the identity map when possible and the $0$ map else.

For $x\in \bt_n$, let $\ps(x)$ be the sum of all elements of the interval $\{ y \in \bt_n \mid
\hmin \leq y \leq x\}$. These sums will be called projective elements.

The projective element $\ps(x)$ is therefore the image of the projective module $\proj(x)$ in the Grothendieck group of the Tamari poset.

\begin{lemma}
  \label{proj_over}
  One has $\ps(x/y)=\ps(x)/\ps(y)$.
\end{lemma}

\begin{proof}
  For every $x$ and $y$, there is a simple bijection
  \begin{equation}
    \begin{cases}
   \{ \hmin \leq a\leq x\} \times \{\hmin \leq b \leq y\}  \simeq \{ \hmin \leq c \leq x\over y\}\\
           (a,b) \mapsto a\over b.
    \end{cases}
  \end{equation}
  The existence of this map follows from properties of elementary
  moves. To define its inverse, one has to check that an elementary
  move starting from $a\over b$ is either of the shape $a\over b \to
  a'\over b$ or of the shape $a\over b \to a\over b'$.
\end{proof}

\begin{lemma}
  \label{proj_under}
  One has $\ps(\Y \under x)=\Y * \ps(x)=\mA \circ_2 \ps(x)$.
\end{lemma}

\begin{proof}
  This follows from the description of such intervals obtained in
  \cite[Prop. 4.1]{int_tamari}.
\end{proof}

\begin{proposition}
  \label{proj_alt}
  A projective element of degree at least $2$ can either be written as
  $P \over Q$ for some projective elements $P$ and $Q$ or as $\mA
  \circ_2 P$ for some projective element $P$.
\end{proposition}
\begin{proof}
  This follows from Lemmas \ref{basic_alt}, \ref{proj_over} and
  \ref{proj_under}
\end{proof}

\begin{proposition}
  \label{proj_1}
  If $P$ and $Q$ are projective elements, then $P \over Q$ is a projective element.
\end{proposition}
\begin{proof}
  If $P=\ps(x)$, $Q=\ps(y)$ then $P\over Q=\ps(x\over y)$ by
  Lemma \ref{proj_over}.
\end{proof}

\begin{proposition}
  \label{proj_2}
   If $P$ is a projective element, then $\mA \circ_2
  P$ is a projective element. 
\end{proposition}
\begin{proof}
  If $P=\ps(x)$ then $\mA \circ_2
  P=\ps(\Y\under x)$ by Lemma \ref{proj_under}.
\end{proof}

\begin{theorem}
  \label{projectives_are_nct}
  Projective elements are noncrossing trees containing the border side $1$.
\end{theorem}

\begin{proof}
  By induction on the degree $n$. This is clearly true if $n=1$, for
  $\gA$ and $\mA$.

  Let $P$ be a projective element of degree $n\geq 2$. One can use
  Prop. \ref{proj_alt}.

  If $P$ can be written $Q/R$ with $Q$ and $R$ projective elements,
  then the induction hypothesis for $Q$ and $R$ implies (using Lemma
  \ref{pas_plante}) that $P$ is a noncrossing tree containing the
  border side $1$.

  Else $P$ can be written $\mA \circ_2 Q$ with $Q$ a projective element. Then
  the induction hypothesis for $Q$ implies (using Lemma
  \ref{pas_plante}) that $P$ is a noncrossing
  tree containing the border side $1$.
\end{proof}

\begin{figure}
  \begin{center}
    \epsfig{file=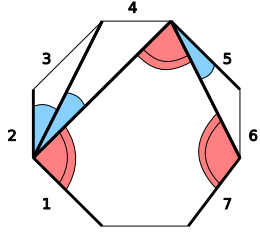,height=3cm} 
    \caption{Angles in a projective noncrossing tree}
    \label{colored}
  \end{center}
\end{figure}

\section{Projective elements and pivots}

\begin{proposition}
  \label{colored_angles}
  In any projective element $P$ of degree $n$, one has
  $N_{\gAs}(P)+N_{\mAs}(P)=n-1$ and $N_{\dAs}(P)=0$.
\end{proposition}
\begin{proof}
  By induction on the degree $n$. This is clearly true if $n=1$, for
  $\gA$ and $\mA$.

  Let $P$ be a projective element of degree $n\geq 2$. One can use
  Prop. \ref{proj_alt}.

  If $P$ can be written $Q/R$ with $Q$ and $R$ projective elements,
  then, by Th. \ref{projectives_are_nct}, angles in $P$ are just
  angles in $Q$, angles in $Q$ and a new angle of type $\gA$. Using
  the induction hypothesis for $Q$ and $R$, one gets the expected
  result.

  Else $P$ can be written $\mA \circ_2 Q$ with $Q$ a projective
  element. In this case, by Th. \ref{projectives_are_nct}, angles in
  $P$ are just angles in $Q$ and a new angle of type $\mA$. Using the
  induction hypothesis for $Q$, one gets the expected result.
\end{proof}

Let us now introduce the notion of \textbf{pivot} from one noncrossing tree $P$
to another one $Q$. One considers an angle of $P$. Let us call
$\alpha$ the common vertex to the two adjacent edges $e$ and $e'$ of
this angle, in trigonometric order. Let $e''$ be the third side of the triangle defined by
this angle. Let $Q$ be the noncrossing tree defined from $P$ by
removing $e$ and adding $e''$. This amounts to rotate clockwise the edge $e$.

\begin{center}
\begin{tikzpicture}
\draw (0,0) node[below left] {$\alpha$} -- node[below] {$e$} (2,0) -- node[right] {$e''$} (1,1.73) -- node[left] {$e'$} (0,0);
\end{tikzpicture}
\end{center}

When noncrossing trees are seen as exceptional collections (see
\cite{araya}), a pivot corresponds to a mutation.

\begin{definition}
  \label{good_pivot}
  Let $P$ and $Q$ be noncrossing trees. If one can go from $P$ to $Q$
  by a pivot replacing an angle of type $\gA$ by an angle of type
  $\mA$, then we will say that $P \to Q$ is a good pivot.
\end{definition}

The following two lemmas are then quite obvious.

\begin{lemma}
  \label{pivot1}
  If $P \to Q$ is a good pivot, then $\mA \circ_2 P \to \mA \circ_2 Q$
  is a good pivot.
\end{lemma}

\begin{lemma}
  \label{pivot2}
  If $P \to Q$ is a good pivot, then $ P \over R \to Q \over R$ is a
  good pivot and $ R \over P \to R \over Q$ is a good pivot.
\end{lemma}

\begin{lemma}
  \label{maps_are_pivots}
  Let $x,y \in \bt_n$. Assume that there is an edge from $y$ to $x$ in
  the Hasse diagram of the Tamari poset. Then $\ps(x) \to \ps(y)$ is a
  good pivot.
\end{lemma}
\begin{proof}
  Edges in the Hasse diagram correspond to elementary moves of planar
  binary trees. Assume first that the elementary move is located at
  the root of a planar binary tree. In the Hasse diagram, this
  corresponds to an edge from $a \over (\Y \under (b \over (\Y \under
  c)))$ to $ a \over (\Y \under b) \over (\Y \under c)$ for some planar binary trees $a,b,c$, possibly empty.

  The corresponding projective elements are $ \ps(a) \over (\Y * \ps(b)) \over (\Y * \ps(c))$ and
$\ps(a) \over (\Y * (\ps(b) \over (\Y * \ps(c))))$.

  In graphical terms, one can see that these projective elements are related by a good pivot as in Fig. \ref{pivot_cover}.

\begin{figure}
  \begin{center}
    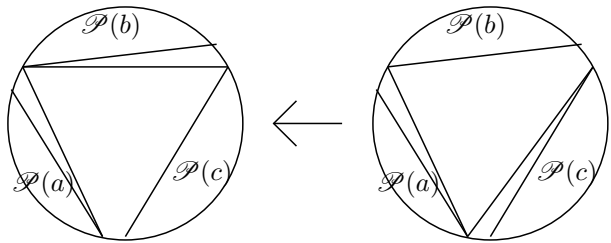
    \caption{The good pivot corresponding to an elementary move}
     \label{pivot_cover}
  \end{center}
\end{figure}

If the elementary move is located higher in the planar binary tree,
then one has to use the description of elementary moves given in \S
\ref{tamari_poset}. Then one concludes by Lemmas \ref{pivot1} and
\ref{pivot2}.

\end{proof}

\begin{lemma}
  \label{pivots_are_maps}
  Let $P$ be a projective element and let $Q$ be a noncrossing tree. If
  $P \to Q$ or $Q \to P$ is a good pivot, then $Q$ is a projective
  element. Moreover, each good pivot between projective elements correspond to an irreducible morphism of
  projective modules.
\end{lemma}
\begin{proof}
  The idea is to count good pivots and irreducible morphisms
  between projective modules. Let $x \in \bt_n$. As the Hasse diagram
  of the Tamari poset is regular of degree $n-1$, there are $n-1$
  edges incident to $x$ in this Hasse diagram. But the Hasse diagram
  also describes irreducible morphisms between projective modules, hence
  there are exactly $n-1$ irreducible morphisms from (or to) $\proj(x)$ to
  (or from) another projective module. By Lemma \ref{maps_are_pivots},
  each edge of the Hasse diagram is given by a good pivot between
  projective elements.

  By Lemma \ref{colored_angles}, there are $n-1$ good pivots from or
  to the projective element $P$. Hence each good pivot corresponds to
  an edge of the Hasse diagram and to an irreducible morphism between
  projective modules.
\end{proof}

From now on, a good pivot from $\ps(x)$ to $\ps(y)$ will (slightly abusing notation) represent also the unique morphism from $\proj(x)$ to $\proj(y)$ that is the identity when possible and $0$ else. From what precedes one gets the following Lemma.

\begin{lemma}
 \label{tout_commute}
 Every two sequences of good pivots between projective modules with common start and common end give the same morphism between projective modules.
\end{lemma}

\section{Other description of projective elements}

\begin{proposition}
  \label{based_proj}
  Any based projective element $P$ of degree at least $2$ can
  be written uniquely as $\gA \circ_1 Q$ for some projective element
  $Q$.
\end{proposition}
\begin{proof}
  By induction on the degree $n$. This is clearly true for the based projective
  element $\gA$. Let us use Prop. \ref{proj_alt}.

  As $P$ is based, it cannot be written $\mA \circ_2 R$ for some
  projective element $R$. Therefore it can be written $ R'\over R$ for
  some projective elements $R$ and $R'$. Necessarily $R$ is based. By
  induction hypothesis, $R$ can be written $\gA \circ_1 R''$ for some
  projective element $R''$. Then $Q= R' \over R''$ is a
  projective element by Prop. \ref{proj_1} and one has $P= \gA \circ_1 Q$.

  Uniqueness is obvious, as there is clearly at most one noncrossing
  tree $Q$ satisfying the hypothesis.
\end{proof}

\begin{proposition}
  \label{proj_as_list}
  Any projective element $P$ has a unique decomposition as a $*$
  product of based projective elements.
\end{proposition}
\begin{proof}
  First, let us note that each noncrossing tree can be uniquely
  written as a product of based noncrossing trees by Lemma
  \ref{decompo_nct}. Therefore uniqueness in the assertion is clear.
  It remains only to prove that all factors are projective elements,
  by induction on the degree of $P$.

  If $P$ is based, then it is its own decomposition.

  If $P$ is not based, let us use Prop. \ref{proj_alt}.

  If $P$ can be written $\mA \circ_2 Q$ for some projective $Q$, then
  one has $P=\Y * Q$. Using the induction hypothesis for $Q$,
  one gets the result for $P$.

  Else $P$ can be written $Q\over R$ for some projective elements $Q$
  and $R$. By induction hypothesis, $R=R_1 * R_2 * \dots * R_k$ where
  $R_i$ are some based projective elements. Then $P=(Q \over R_1) *
  R_2 * \dots * R_k$ by (\ref{star_over}). By Prop. \ref{proj_1}, the
  first factor $Q\over R_1$ is projective. It is also based, hence one
  has obtained the wanted decomposition for $P$.
\end{proof}

\begin{proposition}
  \label{based_do}
  If $P$ is a projective element, then $\gA \circ_1 P$ is a based
  projective element.
\end{proposition}
\begin{proof}
  This follows from Prop. \ref{proj_1}, because $\gA \circ_1 P=P
  \over \Y$.
\end{proof}

\begin{proposition}
  \label{product_do}
  If $P$ and $Q$ are projective elements, then $P*Q$ is a projective element.
\end{proposition}
\begin{proof}
  By induction on $\deg(P)+\deg(Q)$ and $\deg(P)$.

  If $\deg(P)=1$, then $P=\Y$ and $P*Q=\mA \circ_2 Q$ is a
  projective element by Prop. \ref{proj_2}.

  Assume that $\deg(P)\geq 2$. If $P$ is based, then it can be written
  $R\over \Y$ for some projective element $R$, by Prop.
  \ref{based_proj}. Then $P*Q=R \over (\Y*Q)$ by
  (\ref{star_over}). The product $\Y*Q$ is a projective element
  by the initial step of induction. Hence $P*Q$ is a projective
  element by Prop. \ref{proj_1}.

  If $P$ is not based, then $P$ can be written $R*R'$ for some
  projective elements, by Prop. \ref{proj_as_list}. Then
  $P*Q=R*(R'*Q)$ and $R'*Q$ is a projective element by induction on
  the sum of degrees. Therefore $P*Q$ is a projective element by
  induction on the degree of the first factor.
\end{proof}

Let us now give a useful characterization of projective elements.



\begin{proposition}
  \label{proj_vanishing}
  A noncrossing tree $P$ is a projective element if and only if $N_{\dAs}(P)=0$.
\end{proposition}

\begin{proof}
  If $P$ is a projective element, then $N_{\dAs}(P)=0$ by Prop.
  \ref{colored_angles}.

  Assume now that $N_{\dAs}(P)=0$. The proof that $P$ is a projective element 
  uses induction on the degree $n$. This is clear if $n=1$, for $\gA$ and $\mA$.

  If $P$ is based, then it can be written $\gA \circ_1 Q$ for some
  noncrossing tree $Q$ (one uses the hypothesis $N_{\dAs}(P)=0$ to
  show that the right side of $P$ is empty). One then necessarily has
  $N_{\dAs}(Q)=0$. Hence $Q$ is a projective element by induction.
  Therefore $P$ is a projective element by Prop. \ref{based_do}.

  If $P$ is not based, it can be written as a product $R_1*\dots *R_k$
  for some based noncrossing trees $R_1,\dots,R_k$. Then one necessarily has $N_{\dAs}(R_i)=0$ for
  every $i$. Therefore each $R_i$ is projective by induction. Hence
  $P$ is projective by Prop. \ref{product_do}.
\end{proof}

\section{Composition of projective elements}

\begin{proposition}
  \label{circi_proj}
  Let $P,Q$ be projective elements. Assume that $P \circ_i Q$ is a
  noncrossing tree. Then $P \circ_i Q$ is a projective element.
\end{proposition}
\begin{proof}
  One has to distinguish two cases.

  If $P$ contains the border side $i$, then the set of angles of $P \circ_i Q$
  is in bijection with the disjoint union of the set of angles of $P$ and the set of angles of $Q$. By this bijection, the type of each angle is preserved.

  If $P$ does not contain the border side $i$, then necessarily, by Prop. \ref{pas_plante}, $Q$ is based. In this case, there is a bijection between the set of angles of $P \circ_i Q$ and the disjoint union of the set of angles of $P$ and the set of angles of $Q$. By this bijection, the type of each angle is preserved, except maybe for one angle of type $\gA$ of $Q$ between the base side and the border side $1$ of $Q$, which may give an angle of type $\gA$ or $\mA$ in $P \circ_i Q$.

  In both cases, one therefore has $N_{\dAs}(P\circ_i Q)=0$. By Prop. \ref{proj_vanishing}, one gets the result.
\end{proof}

\begin{corollary}
  \label{circ1_proj}
  Let $P,Q$ be projective elements. Then $P \circ_1 Q$ is a projective
  element.
\end{corollary}
\begin{proof}
  This follows from Proposition \ref{circi_proj} and Th.
  \ref{projectives_are_nct}.
\end{proof}

\section{The composition functor $\circ_1$}

Consider the following subset
\begin{equation}
  \ms_{m,1}^n =  \{ (x,y,z) \in \bt_m \times \bt_n \times \bt_{m+n-1} \mid z \in \ps(x) \circ_1 \ps(y)\}.
\end{equation}

Let us now define a module $\tm_{m,1}^n$ over the quiver $\bt_m^{op} \times \bt_n^{op} \times \bt_{m+n-1}$ with all possible commuting relations.

The module $\tm_{m,1}^n$ is given on vertices by a copy of $\QQ$ at each element of $\ms_{m,1}^n$ and the null vector space elsewhere. On the level of maps, it is given by the $\Id$ map whenever possible and the $0$ map else.

One then has to check that relations are satisfied.

\begin{proposition}
  The module $\tm_{m,1}^n$ is a module over the quiver $\bt_m^{op} \times \bt_n^{op} \times \bt_{m+n-1}$ with all commuting relations.
\end{proposition}
\begin{proof}
First, let us show that $\tm_{m,1}^n$ is a $\bt_{m+n-1}$-module.

Indeed, it decomposes (when restricted to the arrows coming from $\bt_{m+n-1}$) as a direct sum over $x$ and $y$, where the component associated with $x,y$ has support $\ps(x) \circ_1 \ps(y)$. By Corollary \ref{circ1_proj}, each such component is a projective $\bt_{m+n-1}$-module. This proves that $\tm_{m,1}^n$ is a projective $\bt_{m+n-1}$-module.

Let us then prove that $\tm_{m,1}^n$ is a $\bt_m^{op} \times \bt_{m+n-1}$-module. 

Let $x \to x'$ be an arrow in the Hasse diagram of $\bt_m$. By Lemma
\ref{maps_are_pivots}, there is a good pivot $\ps(x')\to \ps(x)$. It
follows from the graphical definition of $\circ_1$ on noncrossing
trees that there exists a sequence of good pivots starting from
$\ps(x')\circ_1 \ps(y)$ and ending with $\ps(x)\circ_1
\ps(y)$. Therefore, by Proposition \ref{pivots_are_maps}, there is a
morphism of $\bt_{m+n-1}$-module between the corresponding projective
$\bt_{m+n-1}$-modules.

Furthermore, for any $x,x'\in\bt_m$, any two sequences of good pivots from $\ps(x')\circ_1 \ps(y)$ to $\ps(x)\circ_1 \ps(y)$ give the same map between projective $\bt_{m+n-1}$-modules, by Lemma \ref{tout_commute}.

This implies that $\tm_{m,1}^n$ is a $\bt_m^{op} \times \bt_{m+n-1}$-module.

Let us now prove similarly that $\tm_{m,1}^n$ is a $\bt_n^{op} \times \bt_{m+n-1}$
module.

Let $y \to y'$ be an arrow in the Hasse diagram of $\bt_n$. By Lemma \ref{maps_are_pivots}, there is a good pivot $\ps(y')\to \ps(y)$. It follows from the graphical definition of $\circ_1$ on noncrossing trees that there exists a good pivot from $\ps(x)\circ_1 \ps(y')\to \ps(x)\circ_1 \ps(y)$. Therefore there is a morphism of $\bt_{m+n-1}$-module between the corresponding projective $\bt_{m+n-1}$-modules.

Furthermore, for any $y,y'\in\bt_n$, any two sequences of good pivots from $\ps(x)\circ_1 \ps(y')$ to $\ps(x)\circ_1 \ps(y)$ give the same map between projective $\bt_{m+n-1}$-modules, by Lemma \ref{tout_commute}.

This implies that $\tm_{m,1}^n$ is a $\bt_n^{op} \times \bt_{m+n-1}$-module.

It remains only to prove that $\tm_{m,1}^n$ is a $\bt_n^{op} \times \bt_{m}^{op}$-module. This is again a consequence of Lemma \ref{tout_commute}.

\end{proof}

One can therefore define a composition functor $\circ_1$ from the
category of $\bt_m \times \bt_n$ modules to the category of
$\bt_{m+n-1}$ modules as the tensor product with the module
$\tm_{m,1}^n$.

By definition, the functor $\circ_1$ induces, at the level of the Grothendieck group of the Tamari posets, the composition $\circ_1$.

One consequence is the following.

\begin{proposition}
 \label{circ1_functor}
 The map $\circ_1$ preserves the Euler form of the Tamari posets:
\begin{equation}
 E( x\circ_1 y)=E(x)E(y).
\end{equation}
\end{proposition}
\begin{proof}
  This is an automatic consequence of the existence of the functor
  $\circ_1$, as the Euler form has a natural categorical interpretation.
\end{proof}

\subsection{Other composition functors}

It would be desirable to define the other composition functors
$\circ_i$, for $i>1$. So far, we have not been able to do that directly
as the tensor product with a tri-module. The point is that the
composition maps $\circ_i$ do not preserve the set of projective
elements, unless $i=1$. This makes more difficult to prove the
existence of the necessary tri-module.

There is one indirect way, though, to define these functors.  This
requires first to dispose of an invertible functor which categorifies
$\theta$. Such a functor is given by the
Auslander-Reiten translation on the derived category of the category
of $\bt_n$-modules. From the relation between $\tau$ and $\theta$, one
can then define a functor which categorifies $\tau$. One can use the
axiom (\ref{anticyclic1}) of an anticyclic operad as a model, to
define functors that categorify $\circ_i$.

This gives a possible definition of functors $\circ_i$ between derived
categories. It would be much better to define them at the level of
categories of modules, as the $\circ_i$ maps are known to have good
positivity properties.

In any case, it is enough to have found a functor that categorify the
$\circ_i$ product, to obtain the following result.
\begin{proposition}
 For any $i$, the map $\circ_i$ preserves the Euler form of the Tamari posets:
\begin{equation}
  E( x\circ_i y)=E(x)E(y).
\end{equation}
\end{proposition}

This Proposition can also be deduced directly from
Prop. \ref{circ1_functor}, using the axioms of an anticyclic operad,
and the relation between $\tau$ and $\theta$.

\subsection{Categorification of the $*$ product}

By the same kind of argument as for $\circ_1$, using Prop. \ref{product_do} instead of Corollary \ref{circ1_proj}, one can define a functor $*$ from the
category of $\bt_m \times \bt_n$ modules to the category of $\bt_{m+n}$ modules, that is a categorification of the $*$ product.

This implies the following result.

\begin{proposition}
  The $*$ product respects the Euler form of the Tamari posets: one has $E(x*y)=E(x)*E(y)$.
\end{proposition}

\section{Planar binary trees as noncrossing trees}

Recall that each noncrossing tree is a sum of planar binary trees without multiplicity. Let us now characterize when this sum has only one term.

\begin{lemma}
  A noncrossing tree $P$ is a single planar binary tree if and only if
  $N_{\mAs}(P)=0$.  Moreover, this defines a bijection between simple
  noncrossing trees and planar binary trees.
\end{lemma}
\begin{proof}
  By induction on the degree $n$. This is clearly true if $n=1$, for $\gA$ and $\dA$.

  If $x$ is a planar binary tree of degree $n\geq 2$, then $x$ can be
  written $y \over \Y$, $\Y\under z$ or $y \over \Y \under z$ for some
  smaller planar binary trees $y$ and $z$. By induction hypothesis, $y$ and
  $z$ are noncrossing trees with $N_{\mAs}(y)=N_{\mAs}(z)=0$.
  Therefore $x$ is also a noncrossing tree with $N_{\mAs}(x)=0$.

  Conversely, if $P$ is a noncrossing tree with $N_{\mAs}(P)=0$, then
  $P$ must be based by Lemma \ref{decompo_nct}. Therefore $P$ can be
  written $Q \over \Y$, $\Y \under R$ or $Q \over \Y
  \under R$ for some smaller noncrossing trees $Q$ and $R$. In this case, one
  necessarily has $N_{\mAs}(Q)=0$ and $N_{\mAs}(R)=0$. Therefore each of $Q$ and $R$ is a single planar binary tree by induction. Hence the same is true for $P$.
\end{proof}

Let us call a noncrossing tree $P$ satisfying $N_{\mAs}(P)=0$ a
\textbf{simple noncrossing tree}.

Note that pivots between simple noncrossing trees correspond to
elementary moves between planar binary trees, \textit{i.e.} edges in
the Hasse diagram of the Tamari poset.

\section{The $\#$ product and the $\#$ functor}

On the direct sum $\dend$ of all Abelian groups $\dend(n)$ for $n\geq
1$, there is an associative product $\#$, which can also be defined
using jeu-de-taquin on planar binary trees. The author has learned
about this product from Aval and Viennot, see
\cite{aval_viennot_preprint} and \cite{viennot_fpsac07} for the context
in which they consider the $\#$ product.

Unlike the product $*$, the product $\#$ is graded with respect to the
degree minus $1$, namely it restricts to homogeneous maps
$\dend(m)\otimes \dend(n) \to \dend(m+n-1)$.

In our context, the product $\#$ can be defined as follows:
\begin{equation}
  \label{def_diese}
  x \# y = -\theta^{-1}(\theta(y) \circ_1 \theta(x))
  =-\theta(\theta^{-1}(x) \circ_n \theta^{-1}(y)),
\end{equation}
where $x$ has degree $n$. The equality of the last two terms follows
from the axiom (\ref{anticyclic0}) of anticyclic operads and from the result that
$\tau=(-1)^n \theta^2$ on $\dend(n)$.

\begin{proposition}
  The $\#$ product is associative. The planar binary tree
$\Y$ is a unit for $\#$.
\end{proposition}
\begin{proof}
  The associativity follows from the associativity of $\circ_1$, which
  in turn is a consequence of the axioms of operads. The fact that $\Y$ is a unit follows from the fact that is it the unit of the Dendriform operad.
\end{proof}

\begin{proposition}
  The reversal of planar binary trees is an anti-automorphism of the
  $\#$ product : $\overline{x\# y}=\overline{y} \# \overline{x}$.
\end{proposition}
\begin{proof}
  One uses the equivalent forms of the definition (\ref{def_diese}), the
  relation (\ref{theta_reverse}) between $\theta$ and reversal, and equation (\ref{compo_reverse})
  relating reversal, $\circ_1$ and $\circ_n$.
\end{proof}

\begin{lemma}
  \label{star_diese}
  One has
  \begin{equation}
    (x * y) \# z =x * (y \# z),
  \end{equation}
  and
  \begin{equation}
    (x \# y) * z =x \# (y * z).
  \end{equation}
\end{lemma}
\begin{proof}
  By left-right symmetry, it is enough to prove the second equation.
  By definition (\ref{def_diese}), one has
  \begin{equation}
    x \# (y *z) = -\theta^{-1}(\theta(y*z) \circ_1 \theta(x)).
  \end{equation}
  By (\ref{theta2}), this becomes
  \begin{equation}
     \theta^{-1}((\theta(y)\over \theta(z)) \circ_1 \theta(x)).
  \end{equation}
  By (\ref{circ_over}), this is
  \begin{equation}
     \theta^{-1}((\theta(y)\circ_1 \theta(x))\over \theta(z) ),
  \end{equation}
  which equals, by (\ref{theta3}) and definition (\ref{def_diese}),
  \begin{equation}
    -\theta^{-1}(\theta(y)\circ_1 \theta(x)) * z= (x \# y) *z.
  \end{equation}
\end{proof}

\begin{lemma}
  \label{under_diese}
  One has
  \begin{equation}
    x \# (y \under z) =  (x\# y )\under z. 
  \end{equation}
  and
  \begin{equation}
    (x \over y) \# z = x\over (y \# z). 
  \end{equation}
\end{lemma}
\begin{proof}
  By left-right symmetry, it is enough to prove the first equality.
  Using (\ref{star_circ}), one gets
  \begin{equation}
    (\theta(y)\circ_1 \theta(x))*\theta(z)
    =(\theta(y)*\theta(z))\circ_1 \theta(x).
  \end{equation}
  This becomes, by definition (\ref{def_diese}) and (\ref{theta1}),
  \begin{equation}
    -\theta (x \# y)*\theta(z)
    =-\theta(y \under z) \circ_1 \theta(x).
  \end{equation}
  This is equivalent to
  \begin{equation}
    -\theta^{-1}(\theta (x \# y)*\theta(z))
    =-\theta^{-1}(\theta(y \under z) \circ_1 \theta(x)).   
  \end{equation}
  Therefore we get, by definition (\ref{def_diese}) and (\ref{theta4}),
  \begin{equation}
    (x \# y) \under z= x \# (y \under z).
  \end{equation}
\end{proof}

The set of noncrossing trees is not closed under the $\#$ product, for
instance $\dA \# \gA$ is not a noncrossing tree. One can nevertheless
give a combinatorial description of the $\#$ product of noncrossing
trees, when it is still a noncrossing tree. Let $P$ and $Q$ be
noncrossing trees. If $P$ or $Q$ is not based, then they can be
decomposed as $*$ products and one can use Lemma \ref{star_diese} to
reduce to smaller $\#$ products. Let us now assume that $P$ and $Q$
are based. If $P=P' \over \Y$ or $Q=\Y \under Q'$, then $P \# Q$ can
be described using Lemma \ref{under_diese} as $P'\over Q$ or $P \under
Q'$, and is therefore a noncrossing tree.

\begin{proposition}
  \label{diese_proj}
  The $\#$ product of two projective elements is a projective element.
\end{proposition}
\begin{proof}
  Let $P$ and $Q$ be projective elements. The proof is by induction on
  $\deg(P)$.

  If $P$ has degree $1$, then $P=\Y$, therefore $P \#Q=Q$ and the
  result is true.

  If $P$ is based, then $P=\gA \circ_1 R$ for some projective element
  $R$ by Proposition \ref{based_proj}. Then $P \# Q = (\gA \circ_1 R)
  \# Q=(R \over \Y)\# Q=R \over Q$ by Lemma \ref{under_diese}. This is
  a projective element by Prop. \ref{proj_1}.

  If $P$ is not based, then it can be written $R*R'$ for some
  projective elements $R$ and $R'$ by Proposition \ref{proj_as_list}.
  In this case, $P\# Q= (R*R')\#Q=R *(R'\#Q)$ by Lemma
  \ref{star_diese}. By induction hypothesis, $R' \# Q$ is a projective
  element. Therefore $P\# Q$ is a projective element by Prop.
  \ref{product_do}.
\end{proof}

\begin{proposition}
  The smallest subset of $\dend$ containing $\{\gA,\mA\}$ and stable
  under $\circ_1$ and $\#$ is the set of projective elements.
\end{proposition}
\begin{proof}
  Let us call this set $\mathscr{A}$. Then $\mathscr{A}$ is contained
  in the set of projective elements, because $\gA$ and $\mA$ are
  projective elements and $\circ_1$ and $\#$ preserves projective
  elements by Proposition \ref{circ1_proj} and Proposition \ref{diese_proj}.

  Let us prove the reverse inclusion by induction on the degree $n$.
  This is true if $n=2$. Let $P$ be a projective element of degree at
  least $3$. If $P$ is based, then $P$ can be written $\gA \circ_1 Q$
  for some projective element $Q$, by Proposition \ref{based_proj}. By
  induction, $Q$ is in $\mathscr{A}$, hence $P$ is in $\mathscr{A}$.

  If $P$ is not based, then it can be written $Q*R$ with $Q$ and $R$
  smaller projective elements. If $R=\Y$, then $P=(Q\#\Y)*\Y=Q \# \mA$
  by Lemma \ref{star_diese}. But $Q$ is in $\mathscr{A}$ by induction,
  therefore $P$ too. If $R$ has degree at least $2$, then one has $P=Q
  * (\Y \# R)=(Q * \Y) \# R$ by Lemma \ref{star_diese}. By induction,
  the projective elements $Q*\Y$ and $R$ are in $\mathscr{A}$.
  Therefore $P$ itself is in $\mathscr{A}$.
\end{proof}

It follows from Prop. \ref{diese_proj}, by the same kind of argument as used before for $\circ_1$, that one can define a functor $\#$ from the
category of $\bt_m \times \bt_n$ modules to the category of $\bt_{m+n-1}$ modules, that is a categorification of the $\#$ product.

This implies the following result.

\begin{proposition}
  The $\#$ product respects the Euler form of the Tamari posets: one has $E(x\#y)=E(x)\#E(y)$.
\end{proposition}

The existence of the $\#$ functor also implies that every $\bt_m
\times \bt_n$-module is sent to a $\bt_{m+n-1}$ module. At the level
of the Grothendieck group, this implies that the $\#$ product is
positive on positive elements.

Also one can deduce that the $\#$ product has the following property:
\begin{equation}
  \left(\sum_{s\in \bt_m} s\right) \# \left(\sum_{t\in \bt_n} t\right) = \sum_{u\in \bt_{m+n-1}} u.
\end{equation}

Together with positivity, this implies that the $\#$ product of two sums of planar binary trees without multiplicity is a sum of planar binary trees without multiplicity.

\begin{proposition}
  Let $S$ be a subset of $\bt_m$ and $T$ be a subset of $\bt_n$. Then
  $S \# T$ is the sum over a subset of $\bt_{m+n-1}$.
\end{proposition}

\bibliographystyle{alpha}
\bibliography{dendriformania}

\end{document}